\documentclass[12pt]{amsart}
\usepackage[margin=2.0cm]{geometry}

\usepackage{amssymb}
\usepackage{amsfonts}
\usepackage{mathrsfs}
\usepackage{cite}
\usepackage{graphicx}
\usepackage{mathtools}

\usepackage{float}

\usepackage{todonotes}

\usepackage{pgfplots}
\pgfplotsset{compat=1.15}
\usepackage{mathrsfs}
\usetikzlibrary{arrows}

\definecolor{ffqqqq}{rgb}{1,0,0}
\definecolor{qqzzqq}{rgb}{0,0.6,0}

\newcommand{\R}{{\mathbb R}}

\newtheorem{theorem}{Theorem}[section]

\newtheorem{lemma}[theorem]{Lemma}

\newtheorem{corollary}[theorem]{Corollary}

\DeclareMathOperator{\arcosh}{\mathrm{arcosh}}

\DeclareMathOperator{\artanh}{\mathrm{artanh}}

\DeclareMathOperator{\diam}{\mathrm{diam}}

\DeclareMathOperator{\conv}{\mathrm{conv}}

\DeclareMathOperator{\area}{\mathrm{area}}

\makeatletter
\DeclareFontFamily{U}{tipa}{}
\DeclareFontShape{U}{tipa}{m}{n}{<->tipa10}{}
\newcommand{\arc@char}{{\usefont{U}{tipa}{m}{n}\symbol{62}}}%

\newcommand{\arc}[1]{\mathpalette\arc@arc{#1}}

\newcommand{\arc@arc}[2]{%
  \sbox0{$\m@th#1#2$}%
  \vbox{
    \hbox{\resizebox{\wd0}{\height}{\arc@char}}
    \nointerlineskip
    \box0
  }%
}
\makeatother

\title{On the area of ordinary hyperbolic reduced polygons}
\author{\'Ad\'am Sagmeister}
\address{Alfr\'ed R\'enyi Institute of Mathematics,  Re\'altanoda u. 13-15, H-1053 Budapest, Hungary}
\email{sagmeister.adam@gmail.com }

\mathtoolsset{showonlyrefs=true}

\begin{document}

\begin{abstract}
A convex body $R$ in the hyperbolic plane is reduced if any convex body $K\subset R$ has a smaller minimal width than $R$. We examine the area of a family of hyperbolic reduced $n$-gons, and prove that, within this family, regular $n$-gons have maximal area.
\end{abstract}

\maketitle

\begin{center}
In loving memory of my grandmother.
\end{center}

\section{Introduction}
\label{sec:intro}

The concept of reducedness was introduced by Heil \cite{H78} in 1978 motivated by some volume minimizing problems. A convex body (i. e. a convex compact set of non-empty interior) $K$ is called reduced if an arbitrary convex body $L\subsetneq K$ strictly contained in $K$ has smaller minimal width than $K$. P\'al \cite{Pal} proved in 1921 that for fixed minimal width, the regular triangle has minimal area among convex bodies in the Euclidean plane. This result of his is also known as the isominwidth inequality. The same problem in higher dimensions remains open, as there are no reduced simplices in $\R^n$ for $n\geq 3$ (see \cite{MW02,MS04}), therefore there are no really good candidates for the volume minimizing problems -- so far the best one in $\R^3$ is the so-called Heil body, which has a smaller volume than any rotationally symmetric body of the same minimal width. The problem can be naturally generalized to other spaces, a natural approach is to examine the problem in spaces of constant curvature. Bezdek and Blekherman \cite{Bez00} proved that, if the minimal width is at most $\frac{\pi}{2}$, the regular triangle minimizes the area in $S^2$. However, for spherical bodies of larger minimal width, the minimizers of the isominwidth problem are polars of Reuleaux triangles. Surprisingly enough, there is no solution of the isominwidth problem in the hyperbolic space for arbitrary dimension (this is part of an ongoing work joint with K. J. B\"or\"oczky and A. Freyer).

A reverse isominwidth problem is about finding the maximal volume if the minimal width is fixed. Naturally, this problem does not have a maximizer in general, but for reduced bodies we can ask for the convex body that maximizes the volume. However, in $\R^3$, the diameter of a reduced body of a given minimal width can be arbitrarily large, and hence the Euclidean problem is only interesting on the plane. It is conjectured, that the unique planar reduced bodies maximizing the area and of minimal width $w>0$ in $\R^2$ are the circular disk of radius $\frac{w}{2}$ and the quarter of the disk of radius $w$. A big step towards the proof of this conjecture was made by Lassak, who proved that among reduced $k$-gons, regular ones maximize the area, and as a consequence, all reduced polygons have smaller area than the circle. Following Lassak's footsteps, the same conclusion was derived in $S^2$ by Liu, Chan and Su \cite{LCS}. Interestingly enough, the characterization of hyperbolic reduced polygons is still unclear, but clearly it must be different from the Euclidean and spherical ones; there exist reduced rhombi on the hyperbolic plane, while Euclidean and spherical reduced polygons are all odd-gons (see Lassak \cite{Las90,Las15}). However, the so-called ordinary reduced polygons can be examined the same way (these are odd-gons whose vertices have distance equal to the minimal width of the polygon from the opposite sides such that the projection of the vertices to these sides are in the relative interior of the sides). We have the following main result.

\begin{theorem}
On the hyperbolic plane among ordinary reduced $n$-gons of minimal width $w$, regular $n$-gons of the same minimal width have the greatest area.
\end{theorem}

\section{Preliminaries}

We use the notation $H^2$ for the hyperbolic plane, which is equipped with the geodesic metric. The geodesic distance of two points $x,y\in H^2$ will be denoted as $d\left(x,y\right)$. In this section, we introduce hyperbolic convexity. Many of the concepts are identical with their Euclidean analogues, but as we will soon see, there are exceptions.

For a subset $X$ of the hyperbolic plane $H^2$, we say that $X$ is \emph{convex}, if for any pair of points $x$ and $y$, the unique geodesic segment $\left[x,y\right]$ connecting $x$ and $y$ is a subset of $X$ (where $\left[x,x\right]=\left\{x\right\}$). A \emph{convex body} is a convex compact set of non-empty interior. It is clear, that similarly to Euclidean convexity, the intersection of an arbitrary family of convex sets in the hyperbolic plane is also convex, so we define the \emph{convex hull} of a set $X\subseteq H^2$ as the intersection of all convex sets in $H^2$ containing $X$ as a subset, and we will use the notation $\conv\left(X\right)$ for the convex hull of $X$. The convex body obtained as the convex hull of the finite set $X=\left\{x_1,\ldots,x_n\right\}$ is called a \emph{polygon}, and we use the notation $\left[x_1,\ldots,x_n\right]$ for $\conv\left(X\right)$. A point $x_j\in X$ is a \emph{vertex} of the polygon $X$ if $x_j\not\in\conv\left(X\setminus\left\{x_j\right\}\right)$; a \emph{$k$-gon} in the hyperbolic plane is a polygon of $k$ vertices.

For convex bodies, width is an important concept. On the hyperbolic plane there are many different notions of width (see Santal\'o \cite{S45}, Fillmore \cite{Fil70}, Leichtweiss \cite{Lei05}, Jer\'onimo-Castro--Jimenez-Lopez \cite{JCJL17}, G. Horv\'ath \cite{Hor21}, B\"or\"oczky--Cs\'epai--Sagmeister \cite{BoCsS}, Lassak \cite{Las24}). We will use the width function introduced by Lassak, but we note that it is identical with the extended Leichtweiss width defined by B\"or\"oczky, Cs\'epai and Sagmeister on supporting lines. A hyperbolic line $\ell$ is called a \emph{supporting line} of the convex body $K$ if $K\cap\ell\neq\emptyset$, and $K$ is contained in one of the closed half-spaces bounded by $\ell$. The \emph{width} of the convex body $K$ with respect to the supporting line $\ell$ is the distance of $\ell$ and $\ell'$, where $\ell'$ is a (not necessarily unique) most distant supporting line from $\ell$. It is known that this width function is continuous, and its maximal value coincides with the diameter of the convex body, which will be denoted as $\diam\left(K\right)$. The \emph{minimal width} (i.e. the minimal value of the width function on the set of all supporting lines, also known as the \emph{thickness}) of $K$ is denoted by $w\left(K\right)$. This notion of minimal width is a monotonic function, that is for arbitrary convex bodies $K,L$ such that $K\subseteq L$, we have $w\left(K\right)\leq w\left(L\right)$.  Hence, the concept of hyperbolic reducedness makes perfect sense. A convex body $K$ is called \emph{reduced}, if for any convex body $K'\subsetneq K$, $w\left(K'\right)<w\left(K\right)$. Reduced bodies are well-studied (see Heil \cite{H78}, Lassak--Martini \cite{LaM05,LaM11,LaM14}, Lassak--Musielak \cite{LaM18r}), as they are often extremizers of volume minimizing problems, and bodies of constant width are also reduced.

If we consider the Poincar\'e disk model of the hyperbolic plane $H^2$, hyperbolic lines are either diameters of the unit disk, or circular arcs intersecting the unit circle orthogonally. The boundary points of the unit disk are the \emph{ideal points} of the hyperbolic plane, and hence there is a natural bijection between hyperbolic lines and pairs of ideal points. Besides the identity, there are three types of orientation preserving isometries of the hyperbolic plane depending the number of fixed points. If there is one fixed point, the isometry is called an \emph{elliptic isometry} (or \emph{rotation}). We call an isometry with exactly one fixed ideal point, it is called a \emph{parabolic isometry}. Finally, isometries with exactly two fixed ideal points are called \emph{hyperbolic isometries}, which maps the line corresponding to the two fixed ideal points to itself.

\section{Ordinary reduced polygons}

Lassak proved that hyperpolic odd-gons of thickness $w$ are reduced if all vertices are of distance $w$ from the opposite sides, and the orthogonal projections of these vertices onto the opposite sides are in the relative interior of these sides (see \cite{Las24}). Such polygons are called \emph{ordinary reduced polygons}, since this property characterizes reducedness both in $\R^2$ and in $S^2$ (see Lassak \cite{Las90,Las15}), but not in the hyperbolic plane. In an ongoing work with Ansgar Freyer and K\'aroly Jr. B\"or\"oczky we show that, for each $w>0$ there are reduced rhombi, whose diameters are unbounded. The characterization of hyperbolic reduced polygons is therefore unclear, so we focus on ordinary reduced polygons in this paper. For the diameter of an ordinary reduced polygon of thickness $w$, we have the following by Lassak \cite{Las24}.

\begin{theorem}\label{thm:diameterbound}
Let $R\subset H^2$ be an ordinary reduced polygon of thickness $w$ and diameter $d$. Then,
$$
w<d<\arcosh\left(\cosh w\sqrt{1+\frac{\sqrt{2}}{2}\sinh w}\right).
$$
\end{theorem}

As a consequence, for each $n$ we can expect an $n$-gon of extremal area among ordinary reduced $n$-gons of thickness $w$ by Blaschke's Selection Theorem. In the remainder of the section we will discuss the area of hyperbolic reduced $n$-gons based on the arguments of Lassak \cite{Las05} and Liu--Chang--Su \cite{LCS}.

From now on, $R$ denotes an ordinary reduced $n$-gon in $H^2$ whose vertices are $v_1,\ldots,v_n$ in cyclic order with respect to the positive orientation. For each $i$, let $t_i$ be the orthogonal projection of $v_i$ on the line through $v_{i+\frac{n-1}{2}}$ and $v_{i+\frac{n+1}{2}}$, where the indices are taken mod $n$. By definition, $t_i$ is in the relative interior of $\left[v_{i+\frac{n-1}{2}},v_{i+\frac{n+1}{2}}\right]$, and hence the geodesic segments $\left[v_i,t_i\right]$ and $\left[v_{i+\frac{n+1}{2}},t_{i+\frac{n+1}{2}}\right]$ intersect in a point $p_i$. Let $B_i$ be the union of the two triangles:
$$
B_i=\left[v_i,p_i,t_{i+\frac{n+1}{2}}\right]\cup\left[v_{i+\frac{n+1}{2}},p_i,t_i\right];
$$
we will call $B_i$ a \emph{butterfly}. We observe that these butterflies cover the polygon.

\begin{center}
\includegraphics[scale=0.7]{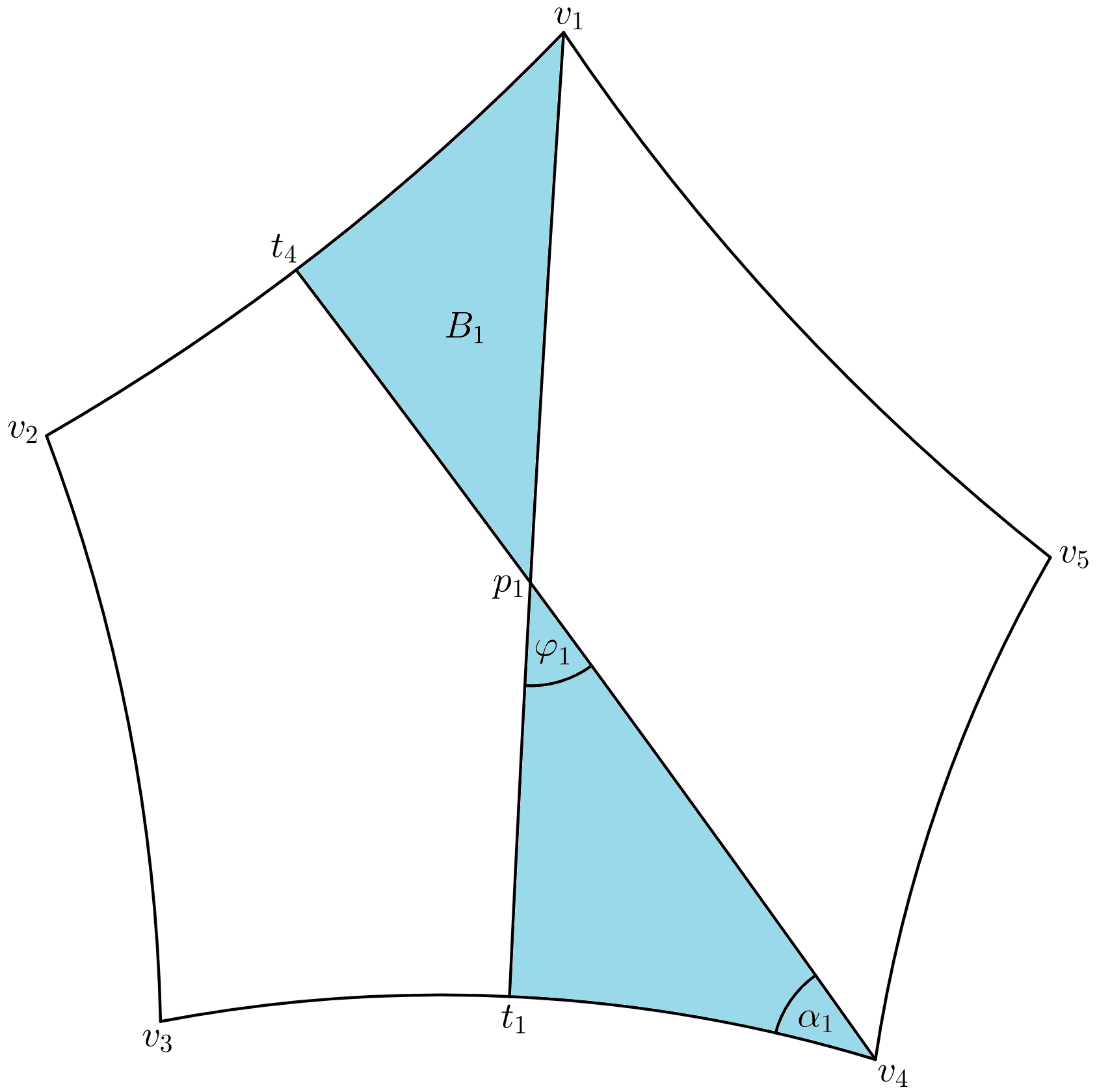}
\end{center}

\begin{lemma}\label{lemma:butterflies}
Let $R\subset H^2$ be an ordinary reduced $n$-gon, and $B_i$ be its $i^{\mathrm{th}}$ butterfly. Then,
$$
R=\bigcup_{i=1}^n B_i.
$$
\end{lemma}

\begin{proof}
Let us denote the line containing $v_i$ and $t_i$ by $\ell_i$.
Since $t_i$ is always in the relative interior of $\left[v_{i+\frac{n-1}{2}},v_{i+\frac{n+1}{2}}\right]$, we have readily $B_i\subset R$ by the convexity of $R$, and hence
$$
R\supseteq\bigcup_{i=1}^n B_i.
$$
Now we want to see that every point of $R$ is covered by some butterfly $B_i$. To prove this, we observe that $B_i$ is the union of the chords obtained by continuously rotating the line $\ell_i$ around $p_i$ into $\ell_{i+\frac{n+1}{2}}$, and then these lines are intersected by $R$. Then we apply a hyperbolic isometry that preserves the line $\ell_{i+\frac{n+1}{2}}$ such that the image of $p_i$ is $p_{i+\frac{n+1}{2}}$. After successively mapping the lines $\ell_1,\ell_{1+\frac{n+1}{2}},\ell_2,\ell_{2+\frac{n+1}{2}},\ldots,\ell_{\frac{n+1}{2}},\ell_1$ by using alternating elliptic and hyperbolic isometries, the composition of these isometries will be a rotation around $p_1$ by $\pi$, and hence
$$
R\subseteq\bigcup_{i=1}^n B_i.
$$
\end{proof}

We introduce a few additional notations for some angles of the butterflies. Let
$$
\varphi=\angle\left(v_i,p_i,t_{i+\frac{n+1}{2}}\right)=\angle\left(t_i,p_i,v_{i+\frac{n+1}{2}}\right)
$$
and
$$
\alpha_i=\angle\left(t_i,v_{i+\frac{n+1}{2}},p_i\right).
$$
The following lemma shows that the two triangles involved in the butterfly $B_i$ are congruent.

\begin{lemma}\label{lemma:congruence}
The two triangles $\left[v_i,p_i,t_{i+\frac{n+1}{2}}\right]$ and $\left[v_{i+\frac{n+1}{2}},p_i,t_i\right]$ defining $B_i$ are congruent.
\end{lemma}

\begin{proof}
It is enough to show that $\angle\left(p_i,t_i,v_{i+\frac{n+1}{2}}\right)=\alpha_i$. To verify this, we show that the triangles $\left[v_i,t_i,v_{i+\frac{n+1}{2}}\right]$ and $\left[v_{i+\frac{n+1}{2}},v_i,t_{i+\frac{n+1}{2}}\right]$ are congruent.

These right triangles share a hypotenuse $\left[v_i,v_{i+\frac{n+1}{2}}\right]$, and furthermore
$$
d\left(v_i,t_i\right)=d\left(v_{i+\frac{n+1}{2}},t_{i+\frac{n+1}{2}}\right)=w.
$$
Thus, by the hyperbolic law of sines we get
$$
\sin\angle\left(t_i,v_{i+\frac{n+1}{2}},v_i\right)=\frac{\sinh w}{\sinh d\left(v_i,v_{i+\frac{n+1}{2}}\right)}=\sin\angle\left(v_i,t_{i+\frac{n+1}{2}},v_{i+\frac{n+1}{2}}\right).
$$
Since the angle sum of a hyperbolic triangle is always less than $\pi$, we deduce
$$
\angle\left(t_i,v_{i+\frac{n+1}{2}},v_i\right)=\angle\left(v_i,t_{i+\frac{n+1}{2}},v_{i+\frac{n+1}{2}}\right).
$$
We recall the following Napier analogy: for a hyperbolic triangle of sides $a,b,c$ and opposite angles $\alpha,\beta,\gamma$ we have
$$
\cot\frac{\gamma}{2}=\tan\frac{\alpha+\beta}{2}\cdot\frac{\cosh\frac{a+b}{2}}{\cosh\frac{a-b}{2}}.
$$
This implies the equality of the angles
$$
\angle\left(v_{i+\frac{n+1}{2},v_i,t_i}\right)=\angle\left(t_{i+\frac{n+1}{2}},v_{i+\frac{n+1}{2}},v_i\right),
$$
so the triangles $\left[v_i,t_i,v_{i+\frac{n+1}{2}}\right]$ and $\left[v_{i+\frac{n+1}{2}},v_i,t_{i+\frac{n+1}{2}}\right]$ are indeed congruent. Finally, we get
\begin{gather*}
\angle\left(t_i,v_i,t_{i+\frac{n+1}{2}}\right)=\angle\left(v_{i+\frac{n+1}{2}},v_i,t_{i+\frac{n+1}{2}}\right)-\angle\left(v_{i+\frac{n+1}{2}},v_i,t_i\right)=\\
=\angle\left(t_i,v_{i+\frac{n+1}{2}},v_i\right)-\angle\left(t_{i+\frac{n+1}{2}},v_{i+\frac{n+1}{2}},v_i\right)=\alpha_i.
\end{gather*}
\end{proof}

\section{Area of ordinary reduced polygons}\label{sec:area}

As a consequence of Lemma~\ref{lemma:congruence}, we can derive the following area formula. Let $F_w\left(x\right)=f_w\left(g_w\left(x\right)\right)$, where
$$
f_w\left(x\right)=\arcsin\frac{x\sqrt{1-\tanh^2 w}}{\tanh w-x}\text{ and }g_w\left(x\right)=\frac{1+\cos x-\sqrt{\left(1+\cos x\right)^2-4\tanh^2 w\cos x}}{2\tanh w}.
$$
\begin{theorem}\label{thm:area}
Let $R\subset H^2$ be an ordinary reduced $n$-gon of thickness $w$ described as above. Then, its area is
$$
\area\left(R\right)=\left(n-2\right)\pi-2\sum_{i=1}^n F_w\left(\varphi_i\right).
$$
\end{theorem}

\begin{proof}
We know from Lemma~\ref{lemma:congruence} and the Gauss--Bonnet formula that the area of $R$ is
$$
\area\left(R\right)=\left(n-2\right)\pi-2\sum_{i=1}^n \alpha_i,
$$
so we only have to calculate $\alpha_i$.
Let $b_i=d\left(p_i,t_i\right)$ and $c_i=d\left(p_i,v_{i+\frac{n+1}{2}}\right)$. From the law of sines, we have
$$
\sin\alpha_i=\frac{\sinh b_i}{\sinh c_i},
$$
so it renains to express $b_i$ and $c_i$ using $\varphi$ and $w$. By Lemma~\ref{lemma:congruence}, we have $c_i=w-b_i$, and then for the right triangle $\left[v_{i+\frac{n+1}{2}},p_i,t_i\right]$ we have
\begin{equation}\label{eq:cosphi}
\cos\varphi_i=\frac{\tanh b_i}{\tanh c_i}=\frac{\tanh b_i}{\tanh\left(w-b_i\right)}=\frac{\tanh b_i\left(1-\tanh w \tanh b_i\right)}{\tanh w-\tanh b_i}. 
\end{equation}
This leads to a quadratic equation for $\tanh b_i$, whose solutions are
$$
\tanh b_i=\frac{1+\cos\varphi_i\pm\sqrt{\left(1+\cos\varphi_i\right)^2-4\tanh^2 w\cos\varphi_i}}{2\tanh w}.
$$
However, it is easy to see that
$$
\frac{1+\cos\varphi_i+\sqrt{\left(1+\cos\varphi_i\right)^2-4\tanh^2 w\cos\varphi_i}}{2\tanh w}>\tanh w,
$$
so from $b_i<w$ we get
\begin{equation}\label{eq:tanhbi}
\tanh b_i=\frac{1+\cos\varphi_i-\sqrt{\left(1+\cos\varphi_i\right)^2-4\tanh^2 w\cos\varphi_i}}{2\tanh w}=g_w\left(\varphi_i\right)
\end{equation}
as the only solution. Using the identity
$$
\sinh\left(\artanh x\right)=\frac{x}{\sqrt{1-x^2}},
$$
we get
$$
\sinh b_i=\frac{g_w\left(\varphi_i\right)}{\sqrt{1-g_w\left(\varphi_i\right)^2}}
$$
and
\begin{gather}\label{eq:sinhci}
\sinh c_i=\sinh\left(\artanh\left(\frac{\tanh w-\tanh b_i}{1-\tanh w \tanh b_i}\right)\right)=\sinh\left(\artanh\left(\frac{\tanh w-g_w\left(\varphi_i\right)}{1-\tanh w \cdot g_w\left(\varphi_i\right)}\right)\right)=\\
=\frac{\frac{\tanh w-g_w\left(\varphi_i\right)}{1-\tanh w \cdot g_w\left(\varphi_i\right)}}{\sqrt{1-\left(\frac{\tanh w-g_w\left(\varphi_i\right)}{1-\tanh w \cdot g_w\left(\varphi_i\right)}\right)^2}}=\frac{\tanh w-g_w\left(\varphi_i\right)}{\sqrt{\left(1-\tanh^2 w\right)\left(1-g_w\left(\varphi_i\right)^2\right)}},
\end{gather}
and hence
$$
\alpha_i=\arcsin\left(\frac{\frac{g_w\left(\varphi_i\right)}{\sqrt{1-g_w\left(\varphi_i\right)^2}}}{\frac{\tanh w-g_w\left(\varphi_i\right)}{\sqrt{\left(1-\tanh^2 w\right)\left(1-g_w\left(\varphi_i\right)^2\right)}}}\right)=f_w\left(g_w\left(\varphi_i\right)\right)=F_w\left(\varphi_i\right).
$$
\end{proof}

Now we are ready to prove our main result.

\begin{theorem}\label{thm:extremum}
Let $R\subset H^2$ be an ordinary reduced $n$-gon of thickness $w$, and let $\widetilde{R}$ be a regular $n$-gon of thickness $w$. Then,
$$
\area\left(R\right)\leq\area\left(\widetilde{R}\right)
$$
with equality if and only if $R$ and $\widetilde{R}$ are congruent.
\end{theorem}

\begin{proof}
Using the notations introduced earlier, Lemma~\ref{thm:area} says that
$$
\area\left(R\right)=\left(n-2\right)\pi-2\sum_{i=1}^n F_w\left(\varphi_i\right)
$$
and
$$
\area\left(\widetilde{R}\right)=\left(n-2\right)\pi-2nF_w\left(\frac{\pi}{n}\right).
$$
Let us show that $F_w$ is concave on the $\left(0,\max_i\varphi_i\right]$ interval. For the sake of simplicity, we introduce
$$
r_w\left(x\right)=\sqrt{\left(1+\cos x\right)^2-4\tanh^2 w\cos x}.
$$
After differentiating, we get
\begin{gather*}
F_w'\left(x\right)=\frac{g_w'\left(x\right)\sqrt{1-\tanh^2 w}\left(\tanh w-g_w\left(x\right)\right)+g_w\left(x\right)\sqrt{1-\tanh^2 w}g_w'\left(x\right)}{\sqrt{1-\left(\frac{g_w\left(x\right)\sqrt{1-\tanh^2 w}}{\tanh w-g_w\left(x\right)}\right)^2}\left(\tanh w-g_w\left(x\right)\right)^2}=\\
=\frac{\tanh w\sqrt{1-\tanh^2 w}g_w'\left(x\right)}{\left(\tanh w-g_w\left(x\right)\right)\sqrt{\left(\tanh w-g_w\left(x\right)\right)^2-\left(1-\tanh^2 w\right) g_w^2\left(x\right)}}=\\
=\frac{-\tanh w\sqrt{1-\tanh^2 w}\sin x}{r_w\left(x\right)\sqrt{\left(\tanh w-g_w\left(x\right)\right)^2-\left(1-\tanh^2 w\right) g_w^2\left(x\right)}}=\\
=\frac{-\sqrt{2}\tanh w\sqrt{1-\tanh^2 w}\sin x}{r_w\left(x\right)\sqrt{\left(1-\cos x\right)\left(-\left(1+\cos x\right)+2\tanh^2 w+r_w\left(x\right)\right)}}.
\end{gather*}
Now we calculate the second derivative. We substitute $r_w'\left(x\right)=\frac{-\sin x}{r_w\left(x\right)}\left(1+\cos x-2\tanh^2 w\right)$ to obtain the following.
\begin{gather*}
F_w''\left(x\right)=\left(-\sqrt{2}\tanh w\sqrt{1-\tanh^2 w}\right)\cdot\\
\frac{\cos x\cdot r_w\left(x\right)\sqrt{1-\cos x}+\sin^2 x\frac{1+\cos x-2\tanh^2 w}{r_w\left(x\right)}\sqrt{1-\cos x}-\sin^2 x\frac{r_w\left(x\right)+\left(1-\cos x\right)}{2\sqrt{1-\cos x}}}{r_w^2\left(x\right)\left(1-\cos x\right)\left(-1-\cos x+2\tanh^2 w+r_w\left(x\right)\right)}=\\
=\left(-\frac{\sqrt{2}}{2}\tanh w\sqrt{1-\tanh^2 w}\right)\cdot\sqrt{\frac{1-\cos x}{-1-\cos x+2\tanh^2 w+r_w\left(x\right)}}\cdot\\
\left(\left(1+\cos x\right)^2-4\tanh^2 w-\left(1+\cos x\right)r_w\left(x\right)\right).
\end{gather*}
We observe that
$$
\left(1+\cos x\right)^2-4\tanh^2 w-\left(1+\cos x\right)r_w\left(x\right)=r_w\left(x\right)\left(r_w\left(x\right)-1-\cos x\right)+4\tanh^2 w\left(\cos x-1\right)<0,
$$
and hence $-F_w$ is strictly concave in the interval $\left(0,\max_i\varphi_i\right]$. Finally, Jensen's inequality concludes the proof.
\end{proof}

\begin{corollary}
The area of a regular $\left(2k+1\right)$-gon of a fixed thickness is a strictly increasing monotonic function of $k$. In particular, all ordinary reduced polygons have a smaller area than the circle of the same thickness.
\end{corollary}

\begin{proof}
If $R_n$ is an ordinary reduced polygon of thickness $w$ with $n$ vertices, and $\widetilde{R}_n$ denotes a regular $n$-gon of thickness $w$, then by Theorem~\ref{thm:extremum},
\begin{gather*}
\area\left(R_n\right)\leq\area\left(\widetilde{R}_n\right)=\left(n-2\right)\pi-2n F_w\left(\frac{\pi}{n}\right)=2n\left(\frac{\pi}{2}-f_w\left(g_w\left(\frac{\pi}{n}\right)\right)\right)-2\pi=\\
=2n\arccos\frac{g_w\left(x\right)\sqrt{1-\tanh^2 w}}{\tanh w-g_w\left(x\right)}-2\pi.
\end{gather*}
From \eqref{eq:cosphi} and \eqref{eq:tanhbi} we get
$$
n=\frac{\pi}{\arccos\frac{g_w\left(\frac{\pi}{n}\right)\left(1-\tanh w\cdot g_w\left(\frac{\pi}{n}\right)\right)}{\tanh w-g_w\left(\frac{\pi}{n}\right)}},
$$
and hence
$$
\area\left(\widetilde{R}_n\right)=2\pi\frac{\arccos\frac{g_w\left(\frac{\pi}{n}\right)\sqrt{1-\tanh^2 w}}{\tanh w-g_w\left(\frac{\pi}{n}\right)}}{\arccos\frac{g_w\left(\frac{\pi}{n}\right)\left(1-\tanh w g_w\left(\frac{\pi}{n}\right)\right)}{\tanh w-g_w\left(\frac{\pi}{n}\right)}}-2\pi.
$$
We observe that
$$
g_w'\left(x\right)=\frac{-\sin x}{r_w\left(x\right)}\left(\tanh w-g_w\left(x\right)\right)<0
$$
for $\tanh w>g_w\left(x\right)$, so $g_w\left(\frac{\pi}{n}\right)$ increases in $n$. Hence, to see that the area of $\widetilde{R}_n$ increases in $n$, it is sufficient to see that the function
$$
h_w\left(x\right)=\frac{\arccos\frac{x\sqrt{1-\tanh^2 w}}{\tanh w-x}}{\arccos\frac{x\left(1-x\tanh w\right)}{\tanh w-x}}
$$
increases for $0<x<g_w\left(0\right)=\frac{1-\sqrt{1-\tanh^2 w}}{\tanh w}$. We note that
$$
h_w'\left(x\right)=\frac{\tanh w \cdot\overline{h}_w\left(x\right)}{\left(\arccos\frac{x\left(1-x\tanh w\right)}{\tanh w-x}\right)^2\left(\tanh w-x\right)\sqrt{\left(1-x^2\right)\left(\tanh^2 w x^2+\tanh^2 w-2x\tanh w\right)}}
$$
where
\begin{gather*}
\overline{h}_w\left(x\right)=\left(x^2+1-2x\tanh w\right)\arccos\frac{x\sqrt{1-\tanh^2 w}}{\tanh w-x}-\\
\sqrt{\left(1-\tanh^2 w\right)\left(1-x^2\right)}\arccos\frac{x\left(1-x\tanh w\right)}{\tanh w-x}.
\end{gather*}
We note that since $0<x<\tanh w$, the sign of $h_w'$ and $\overline{h}_w$ coincide. Differentiating $\overline{h}_w$, we get
$$
\overline{h}'_w\left(x\right)=\frac{x\sqrt{1-\tanh^2 w}}{\sqrt{1-x^2}}\arccos\frac{x\left(1-x\tanh w\right)}{\tanh w-x}-2\left(\tanh w-x\right)\arccos\frac{x\sqrt{1-\tanh^2 w}}{\tanh w-x}.
$$
Observe that
$$
\overline{h}_w\left(\frac{1-\sqrt{1-\tanh^2 w}}{\tanh w}\right)=\overline{h}'_w\left(\frac{1-\sqrt{1-\tanh^2 w}}{\tanh w}\right)=0,
$$
so to verify that $\area\left(\widetilde{R}_n\right)$ increases in $n$, it is enough to see that $\overline{h}''_w\left(x\right)>0$, because then for $0<x<\frac{1-\sqrt{1-\tanh^2 w}}{\tanh w}$, we deduce $\overline{h}'_w\left(x\right)<\overline{h}'_w\left(\frac{1-\sqrt{1-\tanh^2 w}}{\tanh w}\right)=0$, and hence $\overline{h}_w\left(x\right)>\overline{h}_w\left(\frac{1-\sqrt{1-\tanh^2 w}}{\tanh w}\right)=0$. Indeed, we have
\begin{gather*}
\overline{h}''_w\left(x\right)=2\arccos\frac{x\sqrt{1-\tanh^2 w}}{\tanh w-x}+\frac{\sqrt{1-\tanh^2 w}}{\left(1-x^2\right)^{\frac{3}{2}}}\arccos\frac{x\left(1-x\tanh w\right)}{\tanh w-x}+\\
\frac{\tanh w \sqrt{1-\tanh^2 w}}{\left(1-x^2\right)\left(\tanh w\right)\sqrt{\tanh^2 w x^2+\tanh^2 w-2x}}\left(x^3-3x+2\tanh w\right),
\end{gather*}
and that is positive, as $x^3-3x+2\tanh w>0$ if $0<x<\frac{1-\sqrt{1-\tanh^2 w}}{\tanh w}$.

Finally, we observe that
$$
\area\left(R_n\right)<\lim_{m\to\infty}\area\left(\widetilde{R}_m\right)=2\pi\left(\cosh\frac{w}{2}-1\right),
$$
that is the area of a circular disk of radius $\frac{w}{2}$.
\end{proof}

As a final remark, we note that the area of the quarter of the disk (which is also reduced) is greater than the area of the circle of the same thickness, and the author has no knowledge of any reduced body that has greater area in $H^2$.

\noindent{\bf Acknowledgement: } The author wants to thank K\'aroly J. B\"or\"oczky and Ansgar Freyer for their useful comments on the manuscript.

\end{document}